\documentclass[12pt]{amsart}
\usepackage{amsmath,amsfonts,amssymb,a4}
\usepackage{graphicx}
\usepackage{hyperref}
\hypersetup{colorlinks=true,
linkcolor=blue,
citecolor=blue,
pdfauthor={Herbert M\"{o}ller},
pdftitle={Convergence and visualization of Laguerre's rootfinding algorithm}}
\newcommand{\lfrac}[2]{\mbox{\raisebox{-0.6pt}{{\large$\frac{#1}{#2}$}}}}
\newcommand{\lbin}[2]{\mbox{\raisebox{-0.6pt}{{\large$\binom{\;\!#1\;\!}{\;\!#2\;\!}$}}}}

\newtheorem*{theorem}{Theorem}

\begin{document}
\markboth{ \;\; \hrulefill \: Herbert M\"{o}ller
\,\hrulefill}
{\hrulefill \; Convergence and visualization of Laguerre's rootfinding algorithm 
\,\hrulefill \;\quad }
\vspace*{-20mm}
\begin{center}
\Large{\textbf{Convergence and visualization of \\Laguerre's rootfinding algorithm}}\\[7mm]
\large{Herbert M\"{o}ller}\footnote[1]{Math. Institute, Einsteinstr. 62, 48149 
M\"{u}nster\\
\hspace*{5.5mm}Email: herbert.moeller@uni-muenster.de\hfill\copyright\ H. M\"{o}ller 2015}
\end{center}

\emph{Version of January 9, 2015}

\textbf{Abstract} \ Laguerre's rootfinding algorithm is highly recommended although 
most of its properties are known only by empirical evidence. In view of this, we prove 
the first sufficient convergence criterion. It is applicable to simple roots of 
polynomials with degree greater than 3. The ``Sums of Powers Algorithm'' (SPA), 
which is a reliable iterative rootfinding method, can be used to fulfill the 
condition for each root. Therefore, Laguerre's method together with the SPA 
is now a reliable algorithm (LaSPA). In computational mathematics these results 
solve a central task which was first attacked by L. Euler 266 years ago. 

In order to study convergence properties, we eliminate the polynomial and its 
derivatives in the definition of the Laguerre iteration, replacing them by sums, 
only depending on the roots and the iterated values. For this iteration with roots, 
the above criterion of convergence represents an efficient stopping condition. 
In this way, we visualize convergence properties by coloring small neighbourhoods 
of each starting point in squares.


\vspace{-4mm}

\section{The convergence criterion}\label{S:Int}

If $\varrho$ is a simple root of the polynomial $p(z),$ then it has been proved 
that Laguerre's method converges cubically to the limit $\varrho,$ whenever the 
initial guess is close enough to $\varrho.$ But it was not known how small the 
distance from the root must be. In the following theorem we determine for each 
simple root $\varrho$ a disk with centre $\varrho$ such that for all starting 
values in this neighbourhood, the convergence to the root is guaranteed. If all 
roots of $p(z)$ are simple, an a priori lower bound for the radii of all these 
disks is given. At the end of this section, we explain how the condition of the 
theorem can be reliably fulfilled with the SPA.
\begin{theorem}\label{T:Con}
Let $\,p(z)$ be a normalized polynomial with degree $m\ge 4$ and let $\mathrm{P}
{\,:\,=\,}\linebreak
\{\varrho_{1},\ldots,\varrho_{l}\}$ be the set of the roots. For $u_{0}\in
\mathbb{C}\setminus\mathrm{P}$ with $p'(u_{0})\ne 0$ or $p''(u_{0})\ne 0$ the 
\emph{Laguerre sequence} $L_{p}(u_{0})=\,:(u_{n})_{n\in\mathbb{N}}$ is defined 
recursively by
\begin{align}\label{eq:Moe1}
		\begin{split}
				& u_{n+1}=u_{n}-\frac{m}{q(u_{n})+s(u_{n})\,r(u_{n})},\ n\in\mathbb{N},  
				\mbox{\, with \,} q(z):\,=\frac{p'(z)}{p(z)},\\
				& r(z):\,=\sqrt{(m-1)\big(m{\;\!}t(z)-q^2(z)\big)},\ 
				t(z):\,=q^2(z)-\frac{p''(z)}{p(z)} \mbox{\, and}\\
				& s(z):\,={\setlength\arraycolsep{2pt}
				\left\{\begin{array}{rl}1& \text{when $\big(\mathrm{Re}\,q(z)\big)
				\big(\mathrm{Re}\,r(z)\big)+\big(\mathrm{Im}\,q(z)\big)
				\big(\mathrm{Im}\,r(z)\big)>0,$}\\
						-1& \text{else,}
					\end{array}\right.}
		\end{split}
\end{align}
as long as $q(u_{n})+s(u_{n})\,r(u_{n})\ne 0.$

If there is a simple root $\varrho\in \mathrm{P}$ 
such that
\begin{equation}\label{eq:Moe2}
		|u_{0}-\varrho{\:\!}|\le 
\frac{1}{2m-1}\min\big\{|\sigma-\varrho{\:\!}|\,\big|\, 
\sigma\in \mathrm{P}\setminus \{\varrho\}\big\},
\end{equation}
then $L_{p}(u_{0})$ converges to the limit $\varrho,$ and
\begin{equation}\label{eq:Moe3}
		|u_{n}-\varrho{\:\!}|< \lambda^n\,|u_{0}-\varrho{\:\!}| \mbox{\, with \,} 
		\lambda:\,=\frac{15}{16} 
\end{equation}
holds for all $n\in\mathbb{N}\setminus\{0\}$.

If all roots of $p(z)$ are simple, then in \eqref{eq:Moe2} $\mu_{\varrho}:\,=
\min\big\{|\sigma-\varrho{\:\!}|\,\big|\,\sigma\in \mathrm{P}\setminus 
\{\varrho\}\big\}$ can be replaced by the a 
priori lower bound
\[\Big(1+\lfrac{1}{|d_{0}|}\max\limits_{1\le j\le \binom{m}{2}}|d_{j}|
\Big)^{-\frac{1}{2}}<\min\{\mu_{\varrho}\mid \varrho\in\mathrm{P}\},\]
where $d_{j}$ are the coefficients of the polynomial $D_{p}(z):\,=\!
\prod\limits_{1\le i<k\le m}\!\!\big(z-(\varrho_{i}-\varrho_{k})^2\big)
{=\,:}\linebreak
\sum\limits_{j=0}^{\binom{m}{2}}d_{j}z^j.$ These coefficients can be calculated 
only using the coefficients of $p(z).$ All roots of $p(z)$ are simple if and 
only if $d_{0}\ne 0.$
\end{theorem}

\begin{proof}

From $p(z)=\prod\limits_{k=1}^{l}(z-\varrho_{k})^{m_{k}},$ where $m_{k}$ is 
the multiplicity of $\varrho_{k}$ for $k=1,\ldots,l,$ it follows that
\begin{align}\label{eq:Moe4}
		\begin{split}
				& q(z) = \frac{p'(z)}{p(z)}=\mathrm{\frac{d}{dz}}\ln{\:\!}p(z)=
				\mathrm{\frac{d}{dz}}\sum_{k=1}^{l}m_{k}\ln{\:\!}(z-\varrho_{k})=
				\sum_{k=1}^{l}\frac{m_{k}}{z-\varrho_{k}} \mbox{\, and}\\
				& t(z) = q^2(z)-\frac{p''(z)}{p(z)}=-\mathrm{\frac{d}{dz}}{\:\!}q(z)=
				\sum_{k=1}^{l}\frac{m_{k}}{(z-\varrho_{k})^2}.
		\end{split}
\end{align}
Without loss of generality, $\varrho$ can be chosen as $\varrho_{1}$ 
with $m_{1}=1.$ With the abbreviations
\begin{align}\label{eq:Moe5}
		\begin{split}
				& S_{1}=S_{1}(u_{0}):\,=\sum_{k=2}^{l}m_{k}\frac{u_{0}-\varrho}{u_{0}-
				\varrho_{k}},\ S_{2}=S_{2}(u_{0}):\,=\sum_{k=2}^{l}m_{k}\left(\frac{u_{0}-
				\varrho}{u_{0}-\varrho_{k}}\right)^2 \mbox{\, and}\\
				& w=w(u_{0}):\,=\lfrac{1}{m}\Big(1+S_{1}+\tilde{s}\,\sqrt{(m-1)\big(m{\;\!}
				(1+S_{2})-(1+S_{1})^2\big)}{\;\!}\Big),
		\end{split}
\end{align}
 where $\tilde{s}=\tilde{s}(u_{0})\in\{-1, 1\}$ is defined in the same way as 
 $s(u_{0}),$ we get
 \begin{equation}\label{eq:Moe6}
		 |u_{1}-\varrho|=|u_{0}-\varrho|\,\big|1-\lfrac{1}{w}\big|.
	\end{equation}
Since the condition $\big(\mathrm{Re}\,q(z)\big)\big(\mathrm{Re}\,r(z)\big)+
\big(\mathrm{Im}\,q(z)\big)\big(\mathrm{Im}\,r(z)\big)>0$ is equivalent to 
$|q(z)+r(z)|>|q(z)-r(z)|$ and since the latter relation doesn't change when 
each term is multiplied by the same non-zero factor, it follows that 
$\tilde{s}(u_{0})=s(u_{0}).$ First, it will be shown that \eqref{eq:Moe2} 
implies $\tilde{s}(u_{0})=1.$

From \eqref{eq:Moe2} we have $(2m-1)|u_{0}-\varrho|\le |\sigma - \varrho|=
|\sigma -u_{0}+u_{0}-\varrho|\le |u_{0}-\sigma|+|u_{0}-\varrho|$ for each 
$\sigma \in \mathrm{P}\setminus \{\varrho\},$ from which $\big|\lfrac{u_{0}-\varrho}
{u_{0}-\sigma}\big| \le \lfrac{1}{2m-2}$ follows. 
Therefore and with $\sum\limits_{k=2}^{l}m_{k}=m-1$ we get
\begin{equation}\label{eq:Moe7}
		|S_{1}|\le \lfrac{1}{2}  \mbox{\, and \,} |S_{2}|\le \lfrac{1}{4m-4}.
\end{equation}
With $s_{j}:\,=\mathrm{Re}\,S_{j}$ and $t_{j}:\,=\mathrm{Im}\,S_{j},\ j=1,2,$ 
\eqref{eq:Moe7} leads to
\begin{equation}\label{eq:Moe8}
		|s_{1}|\le \lfrac{1}{2},\ |t_{1}|\le \lfrac{1}{2},\  
		|s_{2}|\le \lfrac{1}{4m-4} \mbox{\, and \,} |t_{2}|\le \lfrac{1}{4m-4}.
\end{equation}
We evaluate the square root in \eqref{eq:Moe5} with the well-known formula
\begin{align}\label{eq:Moe9}
		\begin{split}
				& \sqrt{a+i\,b}=\sqrt{\lfrac{1}{2}\bigl(\sqrt{a^2+b^2}+a\bigr)}+i\,
		(\mathrm{sign}\,b)\sqrt{\lfrac{1}{2}\big(\sqrt{a^2+b^2}-a\big)}\\
				& \mbox{\, for \,} a,b\in\mathbb{R}  \mbox{\, and \,} b\ne 0.
		\end{split}
\end{align}
Since 
$\tilde{r}:\,=\sqrt{(m-1)\big(m{\;\!}(1+S_{2})-(1+S_{1})^2\big)}=\,:
\sqrt{(m-1)(c+i{\;\!}d)}$ with $c= \\
m-1+ms_{2}-2s_{1}-s_{1}^2+t_{1}^2$ and 
$d=mt_{2}-2t_{1}-2s_{1}t_{1},$ \eqref{eq:Moe9} gives
\begin{align*}
		\begin{split}
				& \mathrm{Re}\,\tilde{r}=\sqrt{\lfrac{m-1}{2}\big(\sqrt{c^2+d^2}+c\big)}  
		\mbox{\, and \,}\\
				& \mathrm{Im}\,\tilde{r}=(\mathrm{sign}\,d)\sqrt{\lfrac{m-1}{2}
		\big(\sqrt{c^2+d^2}-c\big)}.
		\end{split}
\end{align*}
From $0<m-\lfrac{31}{12}\le c \le m+\lfrac{7}{12}$ and $|d|\le 
\lfrac{11}{6}$ for $m\ge 4,$ we get
\begin{align}\label{eq:Moe10}
		\begin{split}
				& \sqrt{(m-1)\big(m-\lfrac{31}{12}\big)}\le \mathrm{Re}\,\tilde{r} \le 
		\sqrt{(m-1)\big(m+\lfrac{3}{2}\big)}<m+\lfrac{1}{4} 
		\mbox{\, and \,}\\
				& 
				|{\:\!}\mathrm{Im}\,\tilde{r}{\:\!}|=\sqrt{\frac{(m-1)d^2}
				{2\big(\sqrt{c^2+d^2}+c\big)}}\le 
				\frac{11}{12}\sqrt{\frac{m-1}{m-\lfrac{31}{12}}}\le 
				\frac{11}{\sqrt{68}}.
		\end{split}
\end{align}
With the abbreviations $\alpha:\,=\lfrac{16}{25},\ g:\,=\lfrac{43-12\alpha}
{24(1-\alpha^2)},\ h:\,=\lfrac{7}{3(1-\alpha^2)},$ starting from $m\ge 
g+\sqrt{g^2-h}=3.99639\ldots$ and passing $m^2-2gm\ge -h,$ we get from 
\eqref{eq:Moe10}
\begin{equation}\label{eq:Moe11}
		\mathrm{Re}\,\tilde{r}> \frac{16}{25}m-\frac{1}{2}  \mbox{\, for \,} m\ge 4.
\end{equation}
The estimates \eqref{eq:Moe8}, \eqref{eq:Moe10} and \eqref{eq:Moe11} lead to 
$\big(\mathrm{Re}\,(1+S_{1})\big)\big(\mathrm{Re}\ \tilde{r}\big)+
\big(\mathrm{Im}\,(1+S_{1})\big)\big(\mathrm{Im}\ \tilde{r}\big)\ge 
\lfrac{1}{2}\big(\lfrac{16}{25}{\:\!}m-\lfrac{1}{2}\big)-\lfrac{11}{2\sqrt{68}}
> \lfrac{1}{3}$ for $m\ge 4,$ which yields
\[\tilde{s}(u_{0})=s(u_{0})=1.\]
Next, continuing with \eqref{eq:Moe6}, we calculate an upper bound for 
$\vartheta:\,=\big|1-\lfrac{1}{x+i\,y}\big|$ with $x:\,=\mathrm{Re}\,w$ and 
$y:\,=\mathrm{Im}\,w.$ Since
\begin{equation}\label{eq:Moe12}
		\vartheta=\left|1-\frac{x-i\,y}{x^2+y^2}\right|=\sqrt{\Big(1-\frac{x}{x^2+y^2}
		\Big){}^{\rule[-2.5mm]{0mm}{1mm}2}+\Big(\frac{y}{x^2+y^2}\Big)
		{}^{\rule[-2.5mm]{0mm}{1mm}2}}=\sqrt{1-\frac{2x-1}{x^2+y^2}},
\end{equation}
we need a lower bound for $x$ and an upper bound for $x^2+y^2.$ From 
\eqref{eq:Moe8}, \eqref{eq:Moe10} and \eqref{eq:Moe11}, we obtain 
\begin{align*}
		\begin{split}
				& x=\lfrac{1}{m}(1+s_{1}+\mathrm{Re}\,\tilde{r})> 
				\lfrac{1}{m}\big(1-\lfrac{1}{2}+\lfrac{16}{25}m-\lfrac{1}{2}\big)= 
				\lfrac{16}{25},\\
				& x\le \lfrac{1}{m}\big(1+\lfrac{1}{2}+m+\lfrac{1}{4}\big)\le 
				\lfrac{23}{16}  \mbox{\, and \,}\\
				& |y|\le \lfrac{1}{m}\big(t_{1}+|\mathrm{Im}\,\tilde{r}|\big)\le 
				\lfrac{1}{4}\big(\lfrac{1}{2}+\lfrac{11}{\sqrt{68}}\big)< 
				\lfrac{6}{13}  
				\mbox{\, for \,} m\ge 4.
		\end{split}
\end{align*}
Therefore with \eqref{eq:Moe12}, we get
\begin{equation}\label{eq:Moe13}
		\vartheta < \sqrt{\frac{2162577}{2465425}} < \frac{15}{16} = \lambda.
\end{equation}
Now, we prove \eqref{eq:Moe3} by mathematical induction. With $n=0$ and 
$\big|1-\lfrac{1}{w}\big|=\vartheta<\lambda,$ \eqref{eq:Moe6} and 
\eqref{eq:Moe13} result in the basis of the induction
\[|u_{1}-\varrho|<\lambda\,|u_{0}-\varrho|.\]
As inductive step, we show that $|u_{h}-\varrho|<\lambda^h\,|u_{0}-\varrho|$ 
leads to $|u_{h+1}-\varrho|<\lambda^{h+1}\,|u_{0}-\varrho|,$ if $h$ is a 
generic number. Since $\lambda^h<1,$ using \eqref{eq:Moe2}, we have
\begin{equation}\label{eq:Moe14}
		|u_{h}-\varrho{\:\!}|< \lambda^h\,|u_{0}-\varrho{\:\!}|< \frac{1}{2m-1}\,
		\mu_{\varrho}.
\end{equation}
Therefore, in the above derivation, we may replace $u_{0}$ by $u_{h}$ and $u_{1}$ 
by $u_{h+1}.$ Then we get
\[|u_{h+1}-\varrho|<\lambda\,|u_{h}-\varrho|<\lambda^{h+1}\,|u_{0}-\varrho|.\]
Reasoning by induction, it follows that \eqref{eq:Moe3} is valid for all 
$n\in\mathbb{N}\setminus\{0\},$ which means that $L_{p}(u_{0})$ converges to the 
limit $\varrho.$

The a priori lower bound for the distances of the roots of $p(z)$ is 
derived in \cite{Moe:ALA} (p. 379). The coefficients of $D_{p}(z)$ 
are calculated with the aid of the sums of powers
\[\sigma_{j}:\,=\sum\limits_{k=1}^{m}\varrho_{k}^{\;\!j},\ 
j=1,\ldots,(m-1)m  \mbox{\, and \,} \sigma\:\!'_{\!j}:\,=\!\!\sum_{1\le i<k\le 
m}\!\!(\varrho_{i}-\varrho_{k})^{2j},\ j=1,\ldots,\lbin{m}{2}.\]
With the coefficients of 
$p(z)=\,:\sum\limits_{k=0}^{m-1}c_{k}z^k+z^m,$ we have
\[{\setlength{\arraycolsep}{1pt}
\sigma_{k}=\left\{
\begin{array}{rll}&-\!\sum\limits_{j=1}^{k-1}c_{m-j}\sigma_{k-j}-kc_{m-k}\ & 
		\text{for $1\le k\le m$ and}\\
		&-\!\sum\limits_{j=1}^{m}c_{m-j}\sigma_{k-j}& \text{for $m< k\le (m-1)m.$}
\end{array}\right.}\]
Expanding the powers in $\sigma\:\!'_{\!j},$ we get
\[\sigma\:\!'_{\!j}=m\sigma_{2j}+\sum\limits_{l=1}^{j-1}(-1)^l\lbin{2j}{l}
\sigma_{l}\,\sigma_{2j-l}+(-1)^j\lbin{2j-1}{j-1}\sigma_{j}^2  \mbox{\, 
for \,} j=1,\ldots,\lbin{m}{2}.\]
Reversing the first case of the above recursion for $\sigma_{k}$ with 
$\sigma\:\!'_{\!k}$ instead of $\sigma_{k}$ and replacing $c_{m-j}$ 
with $d_{\binom{m}{2}-j},$ we obtain
\[d_{\binom{m}{2}-k}=-\lfrac{1}{k}\sum\limits_{j=0}^{k-1}d_{\binom{m}{2}-j}
\sigma\:\!'_{\!k-j} \mbox{\, for \,} k=1,\ldots,\lbin{m}{2}.\]
The lower one of the estimates in \cite{Moe:ALA} (p. 375, formula (100)) applied 
to $D_{p}(z)$ gives
\[\Big(1+\lfrac{1}{|d_{0}|}\max\limits_{1\le j\le \binom{m}{2}}|d_{j}|
\Big)^{-\frac{1}{2}}<\min\{\mu_{\varrho}\mid \varrho\in\mathrm{P}\}.\]
Since $D_{p}(0)=0$ if and only if $d_{0}=0,$ all roots of $p(z)$ are 
different if $d_{0}\ne 0.$
\end{proof}

These sufficient criteria for the convergence of Laguerre sequences can be 
reliably fulfilled for each root with the aid of the SPA, which was 
introduced in the book \cite{Moe:ALA}. A revised version was published in
\cite{Moe:SP1} and \cite{Moe:SP2}. Therefore we will only sketch how 
the reliability is gained and which of the used methods are also 
valuable for Laguerre's method.

The SPA starts at a present version of D. Bernoulli's method which 
uses quotients of successive sums of powers of the roots to approximate 
dominant roots. We have proved two new basic results for the sequence of 
the quotients which we call ``Bernoulli sequence'', namely, a recursion 
formula and a representation of the minimal absolute value of the roots 
only using the terms of the Bernoulli sequence.

In the case of (nearly) equimodular roots, the second result is combined 
with the search for local minima of the absolute values of the polynomial 
with arguments on a circle which has a good approximation of the minimal 
modulus as radius. All these minima are close to roots, and at least one 
of these roots lies inside the circle which we call ``minimum circle''.

If necessary, this construction is repeated with ``chained minimum 
circles'' or with ``modified Tur\'{a}n circles'' which have radii 
forming a null sequence. Therefore, in any case, after finitely many 
steps a root can be approximated with prescribed accuracy.

Error bounds and stopping criteria are obtained with ``Laguerre disks'' 
\cite{Moe:LAG}
\begin{align*}
		\begin{split}
				& \mathcal{L}_{u}:\,=\{w\in\mathbb{C}\:|\:|w-u+m\,p(u)/\big(2\,p'(u)
\big)|\le m\,|\,p(u)/\big(2\,p'(u)\big)|\}\\
				& \mbox{for \,} u\in\mathbb{C} \mbox{\, with \,} p'(u)\ne 0.
		\end{split}
\end{align*}
Each Laguerre disk contains at least one root $\varrho\in\mathrm{P},$ 
the radius of $\mathcal{L}_{u}$ is smaller than $m|u-\varrho|$ if 
$2m|u-\varrho|\le\min\{\mu_{\varrho}\mid \varrho\in\mathrm{P}\},$ and in this 
case, there is only one root in $\mathcal{L}_{u}.$ With Laguerre disks the SPA 
reliably separates all roots in disjoint disks. Using the distances of such 
disks as lower bounds for $\mu_{\varrho}$ in \eqref{eq:Moe2}, the above Theorem 
guarantees that the SPA can be terminated with Laguerre's method.

Up to now, computer programs for Laguerre's method use values of the 
polynomial in the stopping criteria (see, for example, the NAG 
routine presented in \cite{Moe:MEK} and the C program in 
\cite{Moe:PTVF}), and, if at all, there are only complicated error 
bounds. Now, with the above Theorem and with Laguerre disks, we have 
precise stopping criteria and error bounds also for Laguerre's method.

If a preset bound for the step numbers is exceeded using the very good but 
unproved convergence properties visualized in the next section, then 
the SPA will serve as a ``security net''. In this way, Laguerre's 
method together with the SPA is not only highly efficient but also 
reliable. Therefore we call this combination ``Laguerre and SP Algorithm'' 
(LaSPA). The connection with Euler's investigations in his famous book 
``Introductio in Analysin Infinitorum'' (Introduction to Analysis of the 
Infinite) is explained in \cite{Moe:SP1} and \cite{Moe:SP2}.

\section{Visualization of convergence properties}\label{S:Vcp}

The fundamental theorem of algebra constitutes a one-to-one correspondance 
between finite sets of complex numbers and normalized polynomials with simple 
roots. Therefore, in \eqref{eq:Moe1} we may use the representations of $q(z)$ 
and $t(z)$ by the terminating sums in \eqref{eq:Moe4} to get Laguerre sequences 
only depending on the roots and the iterated values. Since, in that way, it is 
very easy to fulfill the condition of \eqref{eq:Moe2}, we have written a Cython 
program ``Laguerre.pyx'' to visualize convergence properties of Laguerre's 
method in squares which can be freely chosen parallel to the axes in the complex 
plane. Our first goal was to check the statement of the following sentence 
in the introductory paragraph of the lemma ``Laguerre's method'' in the English 
part of Wikipedia: \raisebox{0.5mm}{${\scriptstyle\ll}$}One of the most useful 
properties of this method is that it is, from extensive empirical study, very 
close to being a ``sure-fire'' method, meaning that it is almost guaranteed to 
always converge to \emph{some} root of the polynomial, no matter what initial 
guess is chosen.\raisebox{0.5mm}{${\scriptstyle\gg}$}

The program consists of 150 lines. We will only shortly explain it, because it 
is available in the section ``English subjects'' of \cite{Moe:MK}, and it is 
commented in \cite{Moe:Rep}. The Sage system \cite{Moe:Sage} is needed to run the 
program, because Laguerre.pyx uses graphics from Sage, and Cython, after being 
translated to C, is compiled by the GNU C compiler included in Sage.

The program takes six items as input: a list of different complex numbers (the 
roots), the midpoint of a non-rotated square in the complex plane, the length of 
the sides of this square, the number $m$ of starting values in each row and column 
with equal distances inside the square, a number for the image to be saved, and the 
extension of the image file (eps, pdf, png and 3 more). If 0 is entered as the 
length of the sides, a ``standard square'' will be used which has the same midpoint 
as the smallest non-rotated rectangle containing all roots, and the length of the 
sides is set to the double of the maximal length of the sides of the rectangle.

The program mainly consists of three parts. At the beginning, most of the 
parameters get their initial values, and the squared minimal distance of each 
root from the other roots is calculated. In the central part, for each of the 
$m^2$ starting values the terms of the corresponding Laguerre sequence are 
computed until one of the conditions in \eqref{eq:Moe2} is fulfilled or a preset 
bound for the step number is succeeded. In the first case, the starting value 
is stored in two lists, namely one belonging to the root determined by 
\eqref{eq:Moe2}, and the other assigned to the step number. In the second case, 
the starting value is recorded as an indication for a cyclic Laguerre sequence.

The last part organizes the output which consists of two figures and a list, 
reporting for each occurring step number how many starting values need this 
number of steps to fulfill the condition of \eqref{eq:Moe2}. The figures will be 
explained and analysed using an example with five roots and 250000 starting 
values which we will call ``starting points'' in the following.

Both figures contain five white areas which represent the disks determined by 
\eqref{eq:Moe2}. We call each of these disks ``safe'', because a Laguerre 
sequence with a term in such an area converges safely to the root in the centre 
of the disk. The list of the roots is $[\;\!1.6-0.55j,\ -0.39+0.03j,\ 
-2.32+2.17j,\ 0.2-1.06j,\ -0.02-0.27j\;\!],$ where the terminating $j$ or $J$ 
is the notation of Cython for the imaginary part of complex numbers. All starting 
points in the rest of the figures have coloured neighbourhoods which are touching 
disks in the case of PDF files and small squares with shading colours for the 
pixel graphics of PNG output which can be converted to EPS format.

On a computer with a 2.4 GHz processor, the time for translating to C and compiling 
was 10 seconds, the data were calculated in 8 seconds, and after 110 seconds 
the graphics for two times the starting points outside the safes appeared. 
Since with 90000 starting points the latter time was 50 seconds, the program 
can also be used to produce short motion picture sequences, for example, to 
visualize the effect of moving zeros or for zooming.

The coloured areas in \textsc{Figure} \ref{Moeller1} are named ``limit areas'', 
because the neighbourhood of the starting point of a convergent Laguerre sequence 
gets the colour of the area surrounding the safe which contains the limit of the 
sequence. It is one of the hypotheses formulated afterwards, that these points 
indeed are not isolated.

\begin{figure}[bht]
		\centering
		\includegraphics*{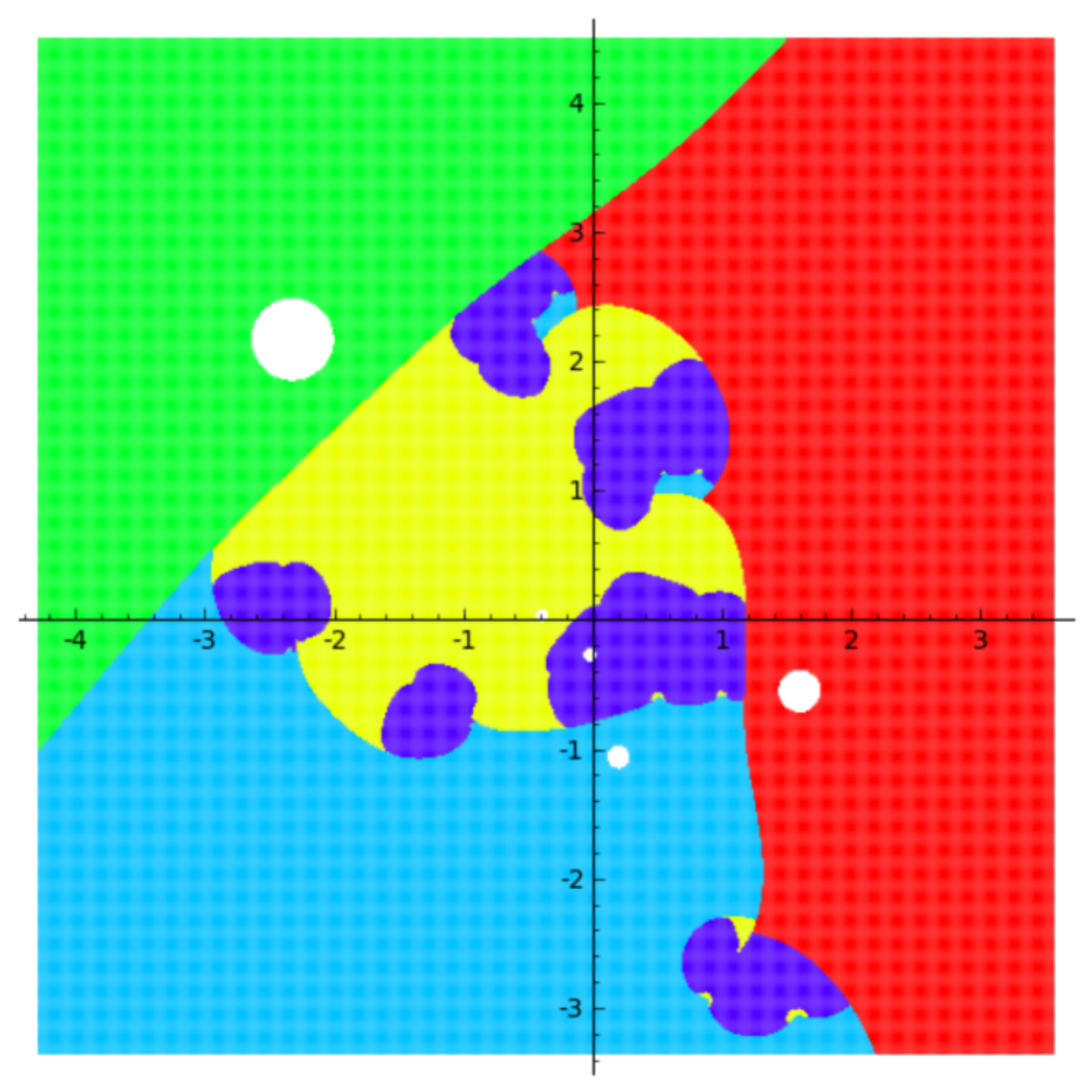}
		\caption{Safes (white disks) and limit areas}
		\label{Moeller1}
\end{figure}

In \textsc{Figure} \ref{Moeller2}, the neighbourhood of the same starting 
points as in \textsc{Figure} \ref{Moeller1} is coloured according to the 
number of steps until the corresponding Laguerre sequence reaches a safe. 
Therefore, the areas with the same colour are called ``step areas''. If the 
occurring colours are listed in the order of the rainbow from red to violet, 
then the index of the colour is the step number. For example, the area 
surrounding a safe is always red, because one step is needed to enter the safe. 
Moreover, the second figure shows a trajectory from a starting point 
to a safe with the maximal number of steps, if the trajectory completely lies in 
the square.\vspace{-1mm}

The arithmetic mean of the roots is given by $\alpha:\,=-\lfrac{1}{m}c_{m-1}.$ 
If it fulfills $\sum\limits_{k=1}^{m}\lfrac{1}{\alpha-\varrho_{k}}\ne 0$ or
$\sum\limits_{k=1}^{m}\lfrac{1}{\alpha-\varrho_{k}}\ne\sum\limits_{k=1}^{m}
\lfrac{1}{(\alpha-\varrho_{k})^2},$ which means that $q(\alpha)+s(\alpha)
\,r(\alpha)\ne 0,$ and if it lies\vspace{-1mm}

\begin{figure}[b]
		\centering
		\includegraphics*{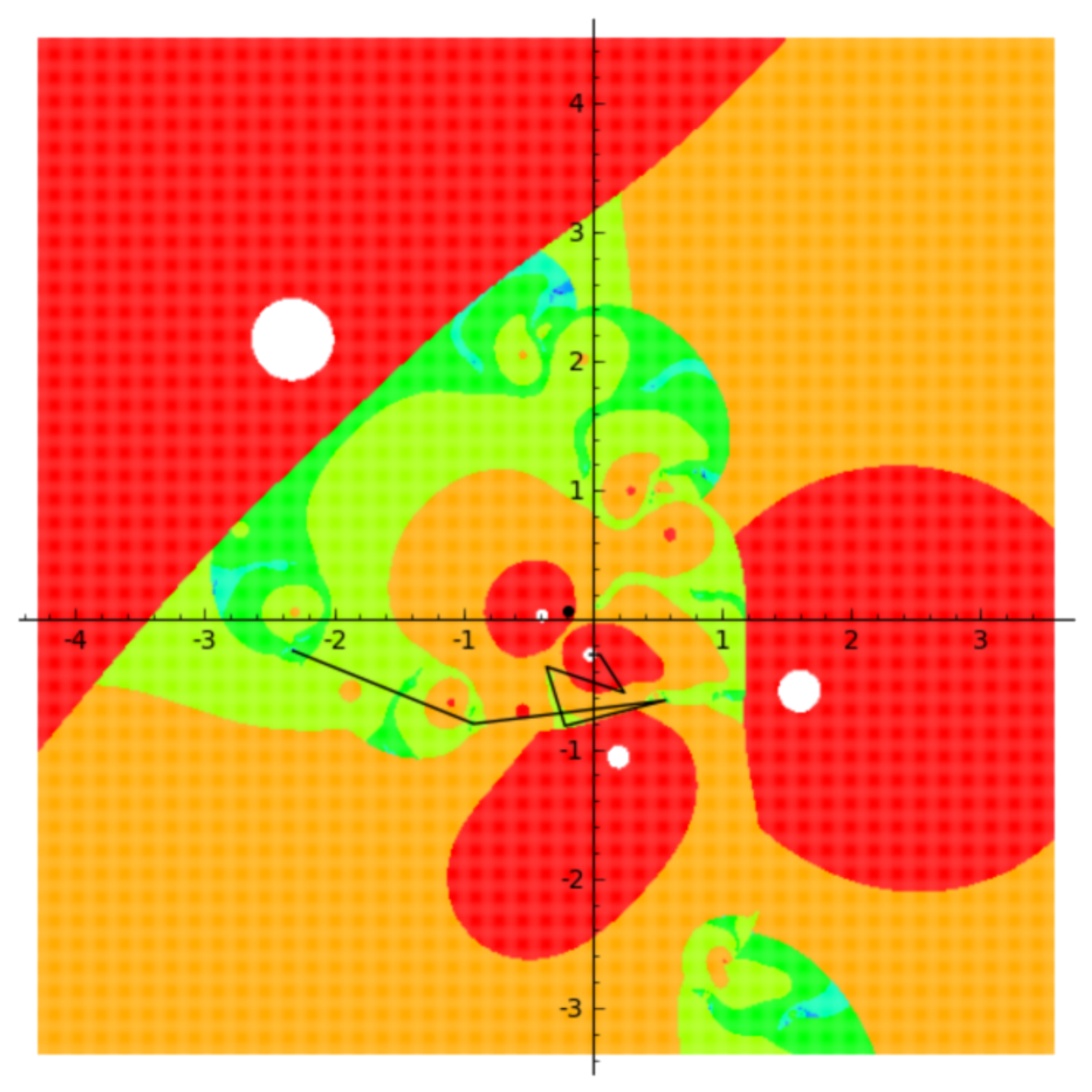}
		\caption{Safes, step areas, a trajectory with maximal step number 7, and
		the arithmetic mean of the roots (black point)}
		\label{Moeller2}
\end{figure}

in the square, then it is shown in the second 
figure as a black point, because in most cases it is an optimal starting 
point.

The possibility of selecting arbitrary squares turns the program into a tool. 
But here, having compared many different figures, we can only state several 
hypotheses and a conjecture concerning the convergence properties of Laguerre's 
method. The corresponding figures with comments and the description of modified 
programs are contained in a report \cite{Moe:Rep}.

\begin{itemize}
		\item  The complex numbers as possible starting points can be divided into 
		four types: inner points, boundary points, ``cycle points'' which generate 
		cyclic Laguerre sequences and ``singularity points'' which as starting points 
		are excluded by definition, because they are the multiple zeros of $p'(z).$ 
		
		\item  The inner points and the boundary points constitute the limit areas and 
		the step areas which both have boundaries consisting of piecewise smooth curves. 
		If there are no singularity points, then a number $S_{p}$ exists such that the 
		starting points of same type with modulus greater than $S_{p}$ form unbounded 
		and simply connected areas. Moreover, each corresponding part of the boundary 
		curve of a limit area or a step area is asymptotic to a straight line.
		
		\item  The cycle points and the singularity points are the only isolated 
		points. 
		
		\item  The boundary curve of each step area belongs completely or piecewise 
		to one of the neighbouring areas.

		\item  If $\zeta$ is a singularity point or a cycle point, then for each 
		number $N\in\mathbb{N}$ there exists a disk $\Delta_{N}$ with centre $\zeta$ 
		such that the step numbers of all Laguerre sequences with $u_{0}\in\Delta_{N}
		\setminus\{\zeta\}$ are greater than $N.$ 
		
		\item  All Laguerre sequences are bounded. This shall be a conjecture, because, 
		possibly, it can be proved with a method similar to that used for the Theorem 
		above. Namely, in this way, it is easy to show that $\lfrac{u_{1}}{u_{0}}$ tends 
		to 0 if $u_{0}$ increases unboundedly.
\end{itemize}

One of the modified programs verifies that it is important to use the sign in the 
denominator of Laguerre's method, because without this condition the limit 
areas and the step areas become highly fragmented. A second program is designed 
for the positioning of very small squares using the list with the occurring 
step numbers and turning off the graphics. In this way,\vspace{1mm} for the polynomial 
$p(z) = z^5$ + (0.93 - 0.32\,$i$)$z^4$ - (0.5818 - 5.8351\,$i$)$z^3$\vspace{1mm} - 
(6.250301 + 1.40811\,$i$)$z^2$ - (1.61432911 + 2.57707253\,$i)z$ + 0.306917325 - 
0.531750585\,$i,$ which\vspace{1mm} corresponds to the\vspace{1mm} above root list,  
the (hypothetical) cycle (0.063005\ldots + 0.1196569\ldots$i$, 
-0.327219\ldots - 0.4305399\ldots$i$) of length 2 is determined.

\end{document}